\documentclass{article}
\usepackage{mathrsfs}
\usepackage{amsfonts}
\usepackage{amssymb,amsmath}
\usepackage{amsthm}
\usepackage[dvips]{graphics}
\usepackage{ae}
\usepackage{bm}
\usepackage{graphicx}
\usepackage{cite}
\usepackage{lineno}
\usepackage{enumitem}
\usepackage{amscd,epsfig,latexsym,mathabx,mathtools,pstricks}

\def\dim{\mbox{\rm dim}}

\def\c{\mbox{\rm c}}
\def\d{\mbox{\rm d}}

\def\And{\mbox{\rm ~and~}}

\def\If{\mbox{\rm ~if~}}
\def\For{\mbox{\rm ~for~}}

\def\rk{\mathrm{rk}}

\def\D{\mbox{\rm D}}

\def\({\mbox{\rm (}}\def\){\mbox{\rm )}}

\def\sp{\hspace{1ex}}
\def\sp{\hspace{0.3cm}}

\newtheorem{theorem}{Theorem}[section]
\newtheorem{proposition}[theorem]{Proposition}
\newtheorem{corollary}[theorem]{Corollary}

\begin{document}
\normalsize
\title{\huge \textbf{Bounds on Characteristic Polynomials}}

\author{Suijie Wang\\
\small Institute of Mathematics, Hunan University, China\\
\small wangsuijie@hnu.edu.cn\\ \\
Yeong-Nan Yeh\\
\small Institute of Mathematics, Academia Sinica, Taiwan\\
\small mayeh@math.sinica.edu.tw\\ \\
Fengwei Zhou\\
\small Department of Mathematics, HKUST, Hong Kong\\
\small fzhou@connect.ust.hk}
\date{}
\maketitle

\begin{abstract}
Suppose $G$ is a simple graph with $n$ vertices, $m$ edges, and rank
$r$. Let $\chi_G(t)=a_0t^n-a_1t^{n-1}+\cdots +(-1)^ra_rt^{n-r}$ be
the chromatic polynomial of $G$. For $q,k\in \Bbb{Z}$ and $0\le k\le
q+r+1$, we obtain a sharp two-side bound for the partial binomial sum of the coefficient sequence, that is,
\[
{r+q\choose k}\le \sum_{i=0}^{k}{q\choose k-i}a_{i}\le {m+q\choose
k}.
\]
Indeed, this bound holds for the characteristic polynomial of hyperplane arrangements and matroids, and its weak version can be generalized to the characteristic polynomial of toric arrangements and arithmetic matroids. We also propose a problem on the geometric interpretation of the above bound.

\end{abstract}

\section{Introduction}
We start with some notations in graph theory. Let $G=(VG, EG)$ be a
simple graph (no loops and multi-edges) with the vertex set $VG$ and
the edge set $EG$. Let $n=|VG|$, $m=|EG|$, and $c$ be the number of
connected components of $G$. Then the rank of $G$ is $r=n-c$. First
appeared in \cite{Birkhoff}, the chromatic polynomial $\chi_G(t)$
counts the number of proper colorings of the graph $G$ with $t$
colors, which can be written as follows,
\[\chi_G(t)=a_0t^n-a_1t^{n-1}+\cdots +(-1)^ra_rt^{n-r}.\]

The chromatic polynomial is one of the most central topics in graph
theory, whose coefficients are mysterious and have caught many
mathematicians' interests. In 1932, Whitney \cite{Whitney2} showed
that the coefficient sequence is sign-alternating, i.e., $a_i>0$.
Moreover, he \cite{Whitney1} gave a combinatorial interpretation to
each coefficient $a_i$, which is equal to the number of those
$i$-subsets of $EG$ that contain no broken circuits, known as the
\emph{no broken circuit} theorem. In 1968, Read \cite{Read} asked
which polynomial is the chromatic polynomial of some graph and
conjectured that the sequence $a_0,a_1,\ldots,a_r$ is unimodal. In
1970, G.H.J. Meredith \cite{Meredith} gave an upper bound for each
coefficient, which is $|a_i|\le {m\choose i}$. In 1974, Dowling and
Wilf \cite{Dowling-Wilson} gave a lower bound for each coefficient
by considering the finite geometric lattices, i.e., $a_i\ge
(m-r){r-1\choose i-1}+{r\choose i}$. For $1\le i\le r$, let the
integral sequence $n_i, n_{i-1}, \ldots$ be obtained from $a_i$ via
the formula $a_i={n_i\choose i}+{n_{i-1}\choose i-1}+\cdots$, and
denote $a_i^{(i/j)}={n_i\choose j}+{n_{i-1}\choose j-1}+\cdots$. In
1976, Wilf \cite{Wilf} obtained that $a_j\ge a_i^{(i/j)}$ for $i>j$.
In 2012, Huh \cite{Huh} gave a positive answer to Read's conjecture
by showing that the coefficient sequence $a_0,a_1,\ldots,a_r$ is
logconcave.

All above results are relatively big steps in the way investigating
the properties of coefficients of the chromatic polynomial of graphs
and little else is known. In this paper, we shall introduce a new
result on the coefficient sequence which will imply Whitney's
sign-alternating result, Meredith's upper bound result, and Dowling
and Wilson's lower bound result.

Next is the statement of our main result. If $q,k\in \Bbb{Z}$ with
$0\le k\le q+r+1$, we have
\begin{equation}\label{main}
{r+q\choose k}\le \sum_{i=0}^{k}{q\choose k-i}a_{i}\le {m+q\choose
k}.
\end{equation}
The most interesting part of above inequalities is when $q$ is a
non-positive integer. Recall the generalized binomial coefficient is
defined by, for $\alpha\in \Bbb{C}$ and $k\in \Bbb{Z}_{\ge 0}$,
\[
{\alpha\choose k}=\frac{1}{k!}\alpha(\alpha-1)\cdots(\alpha-k+1).
\]
Then ${\alpha\choose 0}=1, {0\choose k}=0$, and ${-1\choose
k}=(-1)^{k}$.

When $q=0$, we have
\[
{r\choose k}\le a_{k}\le {m\choose k},\sp \For 0\le k\le r.
\]
It means that Whitney's sign-alternating theorem and Meredith's
upper bound theorem are direct consequences.

If $q=-1$, we have
\[
{r-1\choose k}\le (-1)^k\sum_{i=0}^{k}(-1)^ia_{i}\le {m-1\choose k},
\sp\For 0\le k\le r.
\]
So the first $r-1$ partial sums of the coefficient sequence of the
chromatic polynomial are still sign-alternating.

Taking $q=k-r-1$, we have
\[
(-1)^k\sum_{i=0}^{k}{r-i\choose k-i}(-1)^ia_{i}\ge 0,\sp~\For~ 0\le
k \le r.
\]
Later we shall see that the above inequality implies Dowling and
Wilson's result on the lower bound of the coefficients.

Indeed, all above results hold for a more general object, the
characteristic polynomial of hyperplane arrangements and matroids. Hence in section 2, we shall practise the proof of our main result (\ref{main}) on hyperplane arrangements. In section 3, we will generalize the same result to matroids and give a weak version for arithmetic matroids. Finally, we propose an open problem attempting to give a geometric interpretation of our bounds (\ref{main}) for hyperplane arrangements.

\section{Hyperplane Arrangements}
An $n$-dimensional arrangement $\mathcal{A}$ of hyperplanes is a
finite collection of codimension one subspaces in an $n$-dimensional
vector space $V$. Equipped with the partial order defined by the
inverse of set inclusion, the set of all nonempty intersections of
hyperplanes in $\mathcal{A}$ including the ambient space
$V:=\cap_{H\in \emptyset} H$ forms a semi-lattice $L(\mathcal{A})$,
called the intersection semi-lattice, i.e.,
\[
L(\mathcal{A})=\{\cap_{H\in \mathcal{B}}H\mid \mathcal{B}\subseteq
\mathcal{A}\},
\]
Note that the minimal element of $L(\mathcal{A})$ is $V$. The
maximal rank of the semi-lattice $L(\mathcal{A})$ is called the
\emph{rank} of hyperplane arrangement $\mathcal{A}$, denoted by
$r(\mathcal{A})$. In another viewpoint, the rank $r(\mathcal{A})$ is
the dimension of the vector space spanned by those normal vectors of
hyperplanes in $\mathcal{A}$.  The \emph{characteristic polynomial}
$\chi(\mathcal{A};t)\in \Bbb{C}[t]$ of $\mathcal{A}$ is defined to
be
\[
\chi(\mathcal{A};t):=\sum_{X\in
L(\mathcal{A})}\mu(\hat{0},X)\,t^{{\small\dim(X)}}.
\]
where $\mu$ is the {\emph M\"{o}bius function} of $L(\mathcal{A})$.
Let $G=(VG,EG)$ be a simple graph with the vertex set $VG=[n]$ and
the edge set $EG\subseteq [n]\times [n]$. The \emph{graphic
arrangement} $\mathcal{A}_G$ of $G$ is an $n$-dimensional
arrangement of $|EG|$ hyperplanes whose members are given by
\[
H_{ij}: x_i=x_j,\sp \For (i,j)\in EG.
\]
With these definitions, we have, see Theorem 2.7 in
\cite{Stanley-Vol-1},
\[
\chi(\mathcal{A}_G;t)=\chi_G(t).
\]
It follows that the rank of the graph $G$ is indeed the same as the
rank of the graphic arrangement $\mathcal{A}_G$.

It is well known that the characteristic polynomial satisfies the
\emph{deletion-contraction} recurrence
\[\chi(\mathcal{A};t)=\chi(\mathcal{A}\setminus H_0;t)-\chi(\mathcal{A}/H_0;t),\]
where $H_0\in \mathcal{A}$ is a fixed hyperplane,
$\mathcal{A}\setminus H_0$ is an $n$-dimensional subarrangement of
hyperplanes in $V$ obtained by removing $H_0$ from $\mathcal{A}$,
and $\mathcal{A}/H_0$ is an $(n-1)$-dimensional hyperplane
arrangement in $H_0$ whose members are those restrictions of all
hyperplanes of $\mathcal{A}\setminus H_0$ on $H_0$, i.e.,
\[
\mathcal{A}\setminus H_0=\mathcal{A}\setminus\{H_0\}, \sp
\mathcal{A}/H_0=\{H\cap H_0\mid H\in \mathcal{A}\setminus H_0\}.
\]

A hyperplane arrangement is called \emph{central} if $\cap_{H\in
\mathcal{A}}H\neq \emptyset$. We have $r(\mathcal{A})\le
|\mathcal{A}|$ in general and call $\mathcal{A}$ \emph{boolean} when
$r(\mathcal{A})=|\mathcal{A}|$. It is easy to show that the boolean
hyperplane arrangement is central and its intersection semi-lattice
is isomorphic to the boolean lattice $(2^\mathcal{A},\subseteq)$.
Hence the characteristic polynomial of an $n$-dimensional boolean
arrangement $\mathcal{A}$ of $m$ hyperplanes is
\begin{equation}\label{boolean-formula}
\chi(\mathcal{A};t)=t^{n-m}(t-1)^m=\sum_{i=0}^m(-1)^i{m\choose
i}t^{n-i}.
\end{equation}
The graphic arrangement $\mathcal{A}_G$ is boolean if and only if
the graph $G$ is a forest. Note that a subset $\mathcal{B}$ of
$\mathcal{A}$ naturally defines a subarrangement of hyperplanes in
the same ambient space as $\mathcal{A}$, still denoted $\mathcal{B}$
by abuse of notations. A hyperplane arrangement $\mathcal{A}$ of
rank $r$ is called \emph{in general position} if the subarrangement
$\mathcal{B}$ of $\mathcal{A}$ is boolean whenever $|\mathcal{B}|\le
r$, or not central otherwise. If a central hyperplane arrangement is
in general position if and only if it is boolean. From
\cite[Proposition 2.4]{Stanley-Vol-1}, the characteristic polynomial
of an $n$-dimensional arrangement $\mathcal{A}$ of $m$ hyperplanes
in general position is
\begin{equation}\label{formula-general-position}
\chi(\mathcal{A};t)=t^n-mt^{n-1}+{m\choose
2}t^{n-2}+\cdots+(-1)^r{m\choose r}t^{n-r}.
\end{equation}
Later, we shall prove in Proposition
\ref{proposition-general-position} that the converse of the above
statement is still true by using no broken circuit theorem. First we
need some preparations to state no broken circuit theorem. A subset
$\mathcal{B}$ of the hyperplane arrangement $\mathcal{A}$ is called
\emph{dependent} if $\cap_{H\in\mathcal{B}}H\neq \emptyset$ and
$r(\cap_{H\in\mathcal{B}}H)<|\mathcal{B}|$, i.e., the subarrangement
$\mathcal{B}$ is central but not boolean. Let $\mathcal{A}$ be
totally ordered under a given order $\prec$. A subset of
$\mathcal{A}$ is called a \emph{circuit} if it is a minimal
dependent subset of $\mathcal{A}$. It is obvious that each dependent
subset of $\mathcal{A}$ contains at least a circuit. A \emph{broken
circuit} is a subset of $\mathcal{A}$ obtained by removing the
maximal element from a circuit of $\mathcal{A}$. A subset
$\mathcal{B}$ of $\mathcal{A}$ is called \emph{$\chi$-independent}
if $\cap_{H\in\mathcal{B}}H\neq \emptyset$ and $\mathcal{B}$
contains no broken circuits.

\begin{theorem}[{\bf No Broken Circuit Theorem}{\rm \cite[Theorem 3.55]{Orlik}}]\label{theorem-BC}
Let $\mathcal{A}$ be an $n$-dimensional hyperplane arrangement of
rank $r$ and its characteristic polynomial
\[
\chi(\mathcal{A};t)=a_0t^n-a_1t^{n-1}+\cdots +(-1)^ra_rt^{n-r}.
\]
Then for $0\le k\le r$, $a_k$ is equal to the number of
$\chi$-indenpendent $k$-subsets of $\mathcal{A}$.
\end{theorem}

\begin{proposition}\label{proposition-general-position}Let $\mathcal{A}$ be an $n$-dimensional arrangement of $m$ hyperplanes and $r(\mathcal{A})=r$. Then $\mathcal{A}$ is in general position if and only if its characteristic polynomial is given by the formula $(\ref{formula-general-position})$.
\end{proposition}
\begin{proof}Note that no subset of $\mathcal{A}$ is a circuit if $\mathcal{A}$ is in general position. Then every $k$-subset of $\mathcal{A}$ is $\chi$-independent for all $k\le r$. By Theorem \ref{theorem-BC}, we have $a_i={m\choose i}$.  Conversely, for $0\le k \le r$, $a_k={m\choose k}$  implies that all $k$-subsets of $\mathcal{A}$ are $\chi$-independent. It follows from the definition that if $\mathcal{B}\subseteq \mathcal{A}$ and $|\mathcal{B}|\le r$, $\mathcal{B}$ contains no broken circuits and $\cap_{H\in\mathcal{B}}H\neq \emptyset$. So the subarrangement $\mathcal{B}$ is boolean if $|\mathcal{B}|\le r$. When $k> r$, we have $a_k=0$,  which implies $\cap_{H\in\mathcal{B}}H=\emptyset$ if $|\mathcal{B}|\ge r$.
\end{proof}


Given $\alpha\in \Bbb{C}$, the formal power series $(1+X)^\alpha =
\sum_{k=0}^\infty {\alpha \choose k} X^k$ satisfies that
$(1+X)^\alpha(1+X)^\beta=(1+X)^{\alpha+\beta}$, then we have
\begin{equation}\label{formula-combinatorial-identity}
\sum_{i=0}^{k}{\alpha\choose i}{\beta\choose
k-i}={\alpha+\beta\choose k},\sp \For \alpha,\beta\in \Bbb{C}, ~~
k\in \Bbb{Z}_{\ge 0}.
\end{equation}


\begin{theorem}\label{theorem-main}
Let $\mathcal{A}$ be an $n$-dimensional arrangement of $m$
hyperplanes and its characteristic polynomial
\[
\chi(\mathcal{A};t)=a_0t^n-a_1t^{n-1}+a_2t^{n-2}+\cdots+(-1)^ra_rt^{n-r},
\]
where $r=r(\mathcal{A})$.
If $q,k\in \Bbb{Z}$ satisfies $0\le k\le
q+r+1$, then
\begin{equation}\label{formula-main}
{r+q\choose k}\le \sum_{i=0}^{k}{q\choose k-i}a_{i}\le {m+q\choose
k}.
\end{equation}
\end{theorem}
\begin{proof}First if $\mathcal{A}$ is boolean, then $r=m$. From (\ref{boolean-formula}) and (\ref{formula-combinatorial-identity}), we have
\begin{equation}\label{boolean-case}
\sum_{i=0}^{k}{q\choose k-i}a_{i}=\sum_{i=0}^{k}{q\choose
k-i}{m\choose i}={m+q\choose k}={r+q\choose k}.
\end{equation}
So if $\mathcal{A}$ is boolean, (\ref{formula-main}) holds for any
$q,k\in \Bbb{Z}$. In general, we shall use induction on
$|\mathcal{A}|$ to prove (\ref{formula-main}). Note that if
$|\mathcal{A}|=0$ or $1$, $\mathcal{A}$ is a boolean arrangement.
Suppose the result holds for $|\mathcal{A}|\le m$. Since
(\ref{formula-main}) holds for any boolean hyperplane arrangement,
it is enough to prove the result for the case that
$|\mathcal{A}|=m+1$ and $\mathcal{A}$ is not boolean. In this case,
we have $r=r(\mathcal{A})< |\mathcal{A}|=m+1$, that is to say, the
space spanned by the $m+1$ normal vectors of hyperplanes in
$\mathcal{A}$ has dimension $r<m+1$. So at least one of these $m+1$
normals can be removed without changing the spanning space. In
another word, there is a hyperplane $H_0\in \mathcal{A}$ such that
$r(\mathcal{A}\setminus H_0)=r(\mathcal{A})=r$. Then we can write
\[
\chi(\mathcal{A}\setminus
H_0;t)=b_0t^n-b_1t^{n-1}+b_2t^{n-2}+\cdots+(-1)^rb_rt^{n-r}.
\]
Notice that each maximal element in the intersection semi-lattice
$L(\mathcal{A}/H_0)$ is a maximal element of the intersection
semi-lattice $L(\mathcal{A})$. Since all maximal elements of
$L(\mathcal{A})$ have the same rank $r$, it follows that the rank of
$\mathcal{A}/H_0$ is $r-1$, i.e., $r(\mathcal{A}/H_0)=r-1$. Then we
can write
\[
\chi(\mathcal{A}/H_0;t)=c_0t^{n-1}-c_1t^{n-2}+c_2t^{n-3}+
\cdots+(-1)^{r-1}c_{r-1}t^{n-r}.
\]
Since $|\mathcal{A}\setminus H_0|= m$ and $r(\mathcal{A}\setminus
H_0)= r$, the induction hypothesis implies that, if $0\le k\le
q+r+1$,
\begin{equation}\label{deletion-case}
{r+q\choose k}~\le ~\sum_{i=0}^{k}{q\choose
k-i}b_{i}~\le~{m+q\choose k}.
\end{equation}
Since  $|\mathcal{A}/H_0|\le m$ and $r(\mathcal{A}/ H_0)= r-1$, the
induction hypothesis implies that, if $0\le k-1\le q+r$, i.e., $1\le
k\le q+r+1$,
\begin{equation}\label{restriction-case}
{r-1+q\choose k-1}~\le ~\sum_{i=0}^{k-1}{q\choose
k-1-i}c_{i}=\sum_{i=1}^{k}{q\choose k-i}c_{i-1} ~\le~
{|\mathcal{A}/H_0|+q\choose k-1}.
\end{equation}
Using the deletion-contraction recurrence
$\chi(\mathcal{A};t)=\chi(\mathcal{A}\setminus
H_0;t)-\chi(\mathcal{A}/H_0;t)$, we have
\[
a_0=b_0=1,\sp a_i=b_i+c_{i-1} ~\If~ 1\le i\le r.
\]
It then follows by combining with (\ref{deletion-case}) and
(\ref{restriction-case}) that if $1\le k\le r+q+1$,
\begin{equation}\label{combining}
{r+q\choose k}+{r-1+q\choose k-1}~\le ~\sum_{i=0}^{k}{q\choose
k-i}a_{i}~\le~{m+q\choose k}+{|\mathcal{A}/H_0|+q\choose k-1}.
\end{equation}
Since (\ref{formula-main}) is obviously true when $k=0$, it remains
to show that, if $1\le k\le r+q+1$,
\begin{eqnarray}
{r+q\choose k}+{r-1+q\choose k-1}&\ge& {r+q\choose k},\label{inequality-r}\\
{m+q\choose k}+{|\mathcal{A}/H_0|+q\choose k-1}&\le& {m+q+1\choose
k}\label{inequality-m}.
\end{eqnarray}
Note that ${r-1+q\choose k-1}\ge 0$ if $r-1+q\ge 0$. However if
$r-1+q< 0$, then $1\le k \le r+q+1<2$ implies $k=1$ and
${r-1+q\choose k-1}=1\ge 0$. It completes (\ref{inequality-r}).
Since ${m+q\choose k}+{m+q\choose k-1}= {m+q+1\choose k}$ and
$|\mathcal{A}/H_0|\le m$, (\ref{inequality-m}) is obvious when
$|\mathcal{A}/H_0|+q\ge 0$. Now consider the case
$|\mathcal{A}/H_0|+q<0$. Since $r(\mathcal{A}/H_0)=r-1$, then we
have $|\mathcal{A}/H_0|\ge r-1$. Combining with $1\le k\le r+q+1$,
we obtain that $k=1$. So ${|\mathcal{A}/H_0|+q\choose
k-1}={m+q\choose k-1}$ if $|\mathcal{A}/H_0|+q<0$, which completes
(\ref{inequality-m}).
\end{proof}

By taking $q=0$ in (\ref{formula-main}), Whitney's sign-alternating
theorem and Meredith's upper bound theorem become direct
consequences of Theorem \ref{theorem-main}.
\begin{corollary}\label{corollary-3}
With the same assumptions as {\rm Theorem \ref{theorem-main}}, we
have
\[{r\choose k}\le a_{k}\le {m\choose k},\sp ~\For~ 0\le k \le r.\]
\end{corollary}
Since ${-1 \choose k-i}=(-1)^{k-i}$, after taking $q=-1$ in
(\ref{formula-main}), we obtain two-side bounds for the
partial sums of the coefficient sequence.
\begin{corollary}\label{corollary-4}
With the same assumptions as {\rm Theorem \ref{theorem-main}}, we
have
\[{r-1\choose k}\le (-1)^k\sum_{i=0}^{k}(-1)^ia_{i}\le {m-1\choose k},\sp ~\For~ 0\le k \le r.\]
\end{corollary}
When $k\le r-1$, we have $(-1)^k\sum_{i=0}^{k}(-1)^ia_{i}\ge
{r-1\choose k}\ge 1$, that is to say, the first $r-1$ partial sums
of the coefficient sequence form a sign-alternating sequence.
\begin{corollary}\label{corollary-5}
With the same assumptions as {\rm Theorem \ref{theorem-main}}, we
have
\begin{equation}\label{formula-new}
(-1)^k\sum_{i=0}^{k}{r-i\choose k-i}(-1)^ia_{i}\ge 0,\sp~\For~ 0\le
k \le r.
\end{equation}
\end{corollary}
\begin{proof}
Taking $q=k-r-1$ in (\ref{formula-main}), we have
\[\sum_{i=0}^{k}{k-r-1\choose k-i}a_{i}\ge {k-1\choose k}=0,\sp~\For~ 0\le k \le r.\]
Then (\ref{formula-new}) is obvious since that ${k-r-1\choose
k-i}=(-1)^{k-i}{r-i\choose k-i}$.
\end{proof}

Since the chromatic polynomial $\chi_G(t)$ of a graph $G$ is the
characteristic polynomial of the graphic arrangement
$\mathcal{A}_G$, then the two-side bound $(\ref{formula-main})$
holds for the coefficient sequence of $\chi_G(t)$.
\begin{theorem}\label{theorem-chromatic}
Let $\chi_G(t)=t^n-a_1t^{n-1}+\cdots+(-1)^ra_rt^{n-r}$ be the
chromatic polynomial of a graph $G$ with $n$ vertices, $m$ edges,
and rank $r$. Then the following three statements are equivalent,
\begin{enumerate}[label={\rm (\roman*)}]
\item\label{theorem-i}  $a_k={m\choose k}$ for all $k$ with $1\le k\le r;$
\item\label{theorem-ii} $a_k={r\choose k}$ for all $k$ with $1\le k\le r;$
\item\label{theorem-iii} $G$ is a forest, i.e., $m=r$.
\end{enumerate}
\end{theorem}
\begin{proof}$\ref{theorem-iii}\Rightarrow\ref{theorem-i}$ and $\ref{theorem-iii}\Rightarrow\ref{theorem-ii}$ are easy consequences of Corollary \ref{corollary-3}. Recall $a_1=|EG|=m$, then $\ref{theorem-ii}\Rightarrow\ref{theorem-iii}$ becomes obvious. From Proposition \ref{proposition-general-position}, $a_k={m\choose k}$ if and only if the graphic arrangement $\mathcal{A}_G$ is in general position. Note that the graphic arrangement $\mathcal{A}_G$ is a central hyperplane arrangement. Then $\mathcal{A}_G$ contains no subset of size larger than $r$. Hence we have $|\mathcal{A}_G|=|EG|=m=r$ which proves $\ref{theorem-i}\Rightarrow\ref{theorem-iii}$.
\end{proof}
To end this section, we present an obvious application of Theorem \ref{theorem-main}. Given a root system $\Phi$ in an $n$-dimensional Euclidean space $V$. Let $\Phi^+$ be the set of positive roots of $\Phi$.
The \emph{Coxeter arrangement} of $\Phi$ is defined as
\[
\mathcal{A}(\Phi)=\{H: \alpha \bm x=0\mid \alpha \in \Phi^+\}.
\]
Below are the Coxeter arrangements $\mathcal{A}(A_{n})$,
$\mathcal{A}(B_{n})$, and $\mathcal{A}(D_{n})$ corresponding to types
$A$, $B$, and $D$.
\begin{eqnarray*}
\text{~Hyperplanes of~} \mathcal{A}(A_{n})&:&x_{i}-x_{j}=0 ,\sp
1\leq i<j\leq n.\\
\text{~Hyperplanes of~} \mathcal{A}(B_{n})&:&x_{i}\pm x_{j}=0;
x_{i}=0 ,\sp 1\leq i<j\leq n.\\
\text{~Hyperplanes of~} \mathcal{A}(D_{n})&:&x_{i}\pm x_{j}=0 ,\sp
1\leq i<j\leq n.
\end{eqnarray*}
Their characteristic polynomials are
\begin{eqnarray*}
\chi(\mathcal{A}(A_n); t)&=&t(t-1)\cdots(t-n+1),\\
\chi(\mathcal{A}(B_n); t)&=&(t-1)(t-3)\cdots(t-2n+1),\\
\chi(\mathcal{A}(B_n); t)&=&(t-1)(t-3)\cdots(t-2n+3)(t-n+1)\\
&=&\chi(\mathcal{A}(B_n); t)+n\chi(\mathcal{A}(B_{n-1}); t).
\end{eqnarray*}
The elementary symmetric polynomial of degree $k$ in n variables is defined as
\[
e_k(x_1,\ldots,x_n)=\sum_{I\in {[n]\choose k}}\prod_{i\in I}x_i
\]
Then we can write
\begin{eqnarray*}
\chi(\mathcal{A}(A_n); t)&=&\sum_{i=0}^n(-1)^ie_i(0,1,\ldots,n-1)t^{n-i},\\
\chi(\mathcal{A}(B_n); t)&=&\sum_{i=0}^n(-1)^ie_i(1,3,\ldots,2n-1)t^{n-i},\\
\chi(\mathcal{A}(B_n); t)&=&\sum_{i=0}^n(-1)^i\left(e_i(1,3,\ldots,2n-1)+
ne_{i-1}(1,3,\ldots,2n-3)\right)t^{n-k}.
\end{eqnarray*}
Applying Theorem \ref{theorem-main}, for $q,k\in \Bbb{Z}$ and $0\le k\le
q+r+1$, we have
\begin{eqnarray*}
{n-1+q\choose k}\le \sum_{i=0}^{k}{q\choose k-i}e_i(0,1,\ldots,n-1)\le {\frac{n(n-1)}{2}+q\choose
k},
\end{eqnarray*}
\begin{eqnarray*}
{n+q\choose k}\le \sum_{i=0}^{k}{q\choose k-i}e_i(1,3,\ldots,2n-1)\le {\frac{n(n+1)}{2}+q\choose
k},
\end{eqnarray*}
\begin{eqnarray*}
{n+q\choose k}\le \sum_{i=0}^{k}{q\choose k-i}(e_i(1,3,\ldots,2n-1)+ne_{i-1}(1,3,\ldots,2n-3))
\le {\frac{n(n+1)}{2}+q\choose k}.
\end{eqnarray*}

\section{Matroids and Arithmetic Matroids}

Recall that a \textit{matroid} $\mathfrak{M}_E=(E,\rk)$ is a finite
list of elements $E$ (called \textit{the ground set}) together with
a \textit{rank function} $\rk: 2^E\rightarrow \mathbb{N}\cup \{0\}$
that satisfies the following axioms:
\begin{enumerate}
\item[(R1)] If $A\subseteq E$, then $\rk(A)\leq |A|$.
\item[(R2)] If $A\subseteq B\subseteq E$, then $\rk(A)\leq \rk(B)$.
\item[(R3)] If $A,B\subseteq E$, then $\rk(A\cup B)+\rk(A\cap B)\leq \rk(A)+\rk(B)$.
\end{enumerate}

The notion of a matroid was introduced by H. Whitney~\cite{Whitney3}
to capture the fundamental properties of dependence and independence
that are common to graphs and vectors of matrices. There are a
number of different axiom systems of matroids that are all
equivalent via elementary \textit{cryptomorphisms}. Instead of
writing all equivalent definitions of matroids, we list several
terminologies for matroids, see \cite{White1986,Oxley} for more
details. An element $v\in E$ is \textit{dependent} on $A\subseteq E$
if $\rk(A\cup\{v\})=\rk(A)$, while the element $v$ is
\textit{independent} from $A$ if $\rk(A\cup\{v\})=\rk(A)+1$. A
non-empty sublist $A\subseteq E$ is a \textit{dependent} set if
there exists $v\in A$ such that $v$ is dependent on
$A\setminus\{v\}$, while $A$ is an \textit{independent} set if for
any $v\in A$, $v$ is independent from $A\setminus\{v\}$. We assume
that the empty set is independent. A sublist $B\subseteq E$ is a
\textit{basis} if it is a maximal independent set, while a sublist
$C\subseteq E$ is a \textit{circuit} if it is a minimal dependent
set. An element $v\in E$ is a \textit{coloop} or a \textit{free}
element if $\rk(E\setminus\{v\})=\rk(E)-1$, $v$ is a \textit{loop} or
a \textit{torsion} if $\rk(\{v\})=0$, and $v$ is \textit{proper} if
$v$ is not a loop and $\rk(E\setminus\{v\})=\rk(E)$. It is easy to
check that any element of a matroid is just one of the previous
three types.

There are two classic operations on $\mathfrak{M}_E$,
\textit{deletion} and \textit{contraction}. Given a matroid
$\mathfrak{M}_E=(E,\rk)$ and an element $v\in E$, the
\textit{deletion} of $\mathfrak{M}_E$ with respect to $v$ is the
matroid $\mathfrak{M}_{E\setminus v}=(E\setminus\{v\},\rk_{E\setminus
v})$, where $\rk_{E\setminus v}(A):=\rk(A)$ for all $A\subseteq
E\setminus\{v\}$, while the \textit{contraction} of $\mathfrak{M}_E$
with respect to $v$ is the matroid
$\mathfrak{M}_{E/v}=(E\setminus\{v\},\rk_{E/v})$, where
$\rk_{E/v}(A):=\rk(A\cup\{v\})-\rk(\{v\})$ for all $A\subseteq
E\setminus\{v\}$.

The \textit{Tutte polynomial} $T_E(x,y)=T(\mathfrak{M}_E;x,y)$ of a
matroid $\mathfrak{M}_E=(E,\rk)$ is defined as
$$T_E(x,y):=\sum_{A\subseteq
E}(x-1)^{\rk(E)-\rk(A)}(y-1)^{|A|-\rk(A)}.$$ One of the important
properties of the Tutte polynomial is the
following \textit{deletion-contraction} recurrence relation
\begin{equation}
T_E(x,y)=T_{E\setminus v}(x,y)+T_{E/v}(x,y),
\end{equation}
where $v\in E$ is a proper element of $\mathfrak{M}_E$,
$T_{E\setminus v}(x,y)$ is the Tutte polynomial of
$\mathfrak{M}_{E\setminus v}$, and $T_{E/v}(x,y)$ is the Tutte
polynomial of $\mathfrak{M}_{E/v}$.

The \textit{characteristic polynomial} $\chi(\mathfrak{M}_E;t)$ of a
matroid $\mathfrak{M}_E=(E,\rk)$ is defined as
$$\chi(\mathfrak{M}_E;t):=\sum_{A\subseteq
E}(-1)^{|A|}t^{\rk(E)-\rk(A)}.$$ It is easy to see that
$\chi(\mathfrak{M}_E;t)=(-1)^{\rk(E)}T_E(1-t,0)$ and the
characteristic polynomial satisfies the following
deletion-contraction recurrence relation
\begin{equation}
\chi(\mathfrak{M}_E;t)=\chi(\mathfrak{M}_{E\setminus
v};t)-\chi(\mathfrak{M}_{E/v};t),
\end{equation}
where $v\in E$ is a proper element of $\mathfrak{M}_E$. If one of
the elements of $E$ is a loop, say $u$, then $\chi(\mathfrak{M}_{E};t)=0$ since
$\rk(A\cup\{u\})=\rk(A)$ for all sublists $A\subseteq
E\setminus\{u\}$. If all elements of
$E$ are coloops, then
\begin{equation}
\chi(\mathfrak{M}_E;t)=\sum_{i=0}^{m}\sum_{|A|=i}(-1)^it^{m-i}=\sum_{i=0}^{m}(-1)^i{m\choose
i}t^{m-i},
\end{equation}
where $m=|E|$.

Note that the set of all normal vectors of a hyperplane arrangement forms a representable matroid under the linear dependence and independence relation of vector space. Recall that all arguments used in Theorem \ref{theorem-main} are not much more than deletion-contraction recurrence and induction on the number of hyperplanes in the arrangement. Hence, we are able to obtain the matroid versions of Theorem \ref{theorem-main} and other consequences by those arguments, which we omit here.

\begin{theorem}
Let $\mathfrak{M}_E=(E,\rk)$ be a matroid without loops and its
characteristic polynomial
$$\chi(\mathfrak{M}_E;t)=a_0t^r-a_1t^{r-1}+a_2t^{r-2}+\cdots+(-1)^ra_r,$$
where $r=\rk(E)$. If $q,k\in \Bbb{Z}$ satisfies $0\le k\le q+r+1$,
then
\begin{equation}\label{matroid-1}
{r+q\choose k}\le \sum_{i=0}^{k}{q\choose k-i}a_{i}\le
{m+q\choose k},
\end{equation}
where $m=|E|$. In particular, taking $q=-1$ and $q=k-r-1$ respectively, we have
\begin{equation}\label{matroid-2}
{r-1\choose k}\le (-1)^k\sum_{i=0}^{k}(-1)^ia_{i}\le {m-1\choose k},\sp ~\For~ 0\le k \le r,
\end{equation}
and
\begin{equation}\label{matroid-3}
(-1)^k\sum_{i=0}^{k}{r-i\choose k-i}(-1)^ia_{i}\ge 0,\sp~\For~ 0\le
k \le r.
\end{equation}
\end{theorem}

In 1974, Dowling and Wilson \cite[Theorem 2]{Dowling-Wilson} found a
lower bound for the coefficients of the characteristic polynomial of
a finite geometric lattice, the flag lattice of a matroid. Here we shall recover it by applying (\ref{matroid-3}).
\begin{theorem}{\rm \cite{Dowling-Wilson}}
Let $\mathfrak{M}_E=(E,\rk)$ be a matroid without loops and its
characteristic polynomial
\[
\chi(\mathfrak{M}_E;t)=a_0t^r-a_1t^{r-1}+a_2t^{r-2}+
\cdots+(-1)^ra_r,
\]
where $r=\rk(E)$.
Then
\begin{equation}\label{bound-new}
a_k\ge (m-r){r-1\choose k-1}+{r\choose k},\sp~\For~ 0\le k \le r.
\end{equation}
\end{theorem}
\begin{proof}Applying (\ref{matroid-3}), we have
\begin{eqnarray*}
a_k+(-1)^k\sum_{i=0}^{k-1}{r-i\choose
k-i}(-1)^ia_{i}=(-1)^k\sum_{i=0}^{k}{r-i\choose k-i}(-1)^ia_{i}\ge 0
\end{eqnarray*}
which implies
\begin{eqnarray*}
a_{k}
&\ge&(-1)^{k-1}\sum_{i=0}^{k-1}{r-i\choose k-i}(-1)^ia_{i}\\
&=&(r-k+1)(-1)^{k-1}\sum_{i=0}^{k-1}{r-i\choose
k-1-i}\left(1-\frac{k-1-i}{k-i}\right)(-1)^ia_{i}.
\end{eqnarray*}
Applying (\ref{matroid-3}), we have
\begin{eqnarray*}
a_{k}
&\ge& (-1)^{k}(r-k+1)\sum_{i=0}^{k-1}{r-i\choose k-1-i}\frac{k-1-i}{k-i}(-1)^ia_{i}\\
&=&(-1)^{k}(r-k+1)\frac{(r-k+2)}{2}\sum_{i=0}^{k-2}{r-i\choose k-2-i}\left(1-\frac{k-2-i}{k-i}\right)(-1)^ia_{i}\\
&\ge& (-1)^{k-1}(r-k+1)\frac{(r-k+2)}{2}\sum_{i=0}^{k-2}{r-i\choose k-2-i}\frac{k-2-i}{k-i}(-1)^ia_{i}\\
&\ge &(-1)^{k-2}(r-k+1)\frac{(r-k+2)}{2}\frac{(r-k+3)}{3}
\sum_{i=0}^{k-3}{r-i\choose k-3-i}\frac{k-3-i}{k-i}(-1)^ia_{i}\\
&&\cdots\cdots\\
&\ge& (-1)^{3}{r-2\choose k-2}\sum_{i=0}^{2}{r-i\choose
2-i}\frac{2-i}{k-i}(-1)^ia_{i}.
\end{eqnarray*}
Since $a_0=1$ and $a_1=m$, then
\begin{eqnarray*}
a_{k}\geq m{r-1\choose k-1}-(k-1){r\choose k}=(m-r){r-1\choose
k-1}+{r\choose k}.
\end{eqnarray*}
\end{proof}

Let $R\subseteq S\subseteq E$ be sublists of $E$. Then
$[R,S]:=\{A:R\subseteq A\subseteq S\}$ is a \textit{molecule} if $S$
is the disjoint union $S=R\cup F\cup T$ and for each $A\in [R,S]$,
$$\rk(A)=\rk(R)+|A\cap F|.$$ An \textit{arithmetic matroid} $\mathfrak{A}_E=(\mathfrak{M}_E, m)=(E,\rk,m)$ is a matroid $\mathfrak{M}_E$ equipped with a \textit{multiplicity function} $m: 2^E\rightarrow \mathbb{N}$ satisfying the following axioms:
\begin{enumerate}
\item[(M1)] If $A\subseteq E$ and $v\in E$ is dependent on $A$, then $m(A\cup \{v\})$ divides $m(A)$.
\item[(M2)] If $A\subseteq E$ and $v\in E$ is independent from $A$, then $m(A)$ divides $m(A\cup \{v\})$.
\item[(M3)] If $[R,S]$ is a molecule, then $$m(R)m(S)=m(R\cup F)m(R\cup T).$$
\item[(M4)] If $R\subseteq S\subseteq E$ and $\rk(R)=\rk(S)$, then $$\rho(R,S):=\sum_{A\in[R,S]}(-1)^{|A|-|R|}m(A)\geq 0.$$
\item[(M5)] If $R\subseteq S\subseteq E$ and $|R|+\rk(R^c)=|S|+\rk(S^c)$, then $$\rho^{\ast}(R,S):=\sum_{A\in[R,S]}(-1)^{|A|-|R|}m(A^c)\geq 0.$$
\end{enumerate}
The notion of an arithmetic matroid was introduced by M. D'Adderio
and L. Moci~\cite{Moci2011} to axiomatize both the linear algebra
and the arithmetic of a list of elements in a finitely generated
abelian group. It can be seen as a generalization of the notion of a
matroid.

Also, We can consider the \textit{deletion} and \textit{contraction}
of arithmetic matroids. Given an arithmetic matroid
$\mathfrak{A}_E=(\mathfrak{M}_E,m)=(E,\rk,m)$ and an element $v\in
E$, the \textit{deletion} of $\mathfrak{A}_E$ with respect to $v$ is
the pair $\mathfrak{A}_{E\setminus v}=(\mathfrak{M}_{E\setminus
v},m_{E\setminus v})$, where $\mathfrak{M}_{E\setminus
v}=(E\setminus\{v\},\rk_{E\setminus v})$ is the deletion of
$\mathfrak{M}_E$ and $m_{E\setminus v}(A):=m(A)$ for all $A\subseteq
E\setminus\{v\}$, while the \textit{contraction} of $\mathfrak{A}_E$
with respect to $v$ is the pair
$\mathfrak{A}_{E/v}=(\mathfrak{M}_{E/v},m_{E/v})$, where
$\mathfrak{M}_{E/v}=(E\setminus\{v\},\rk_{E/v})$ is the contraction
of $\mathfrak{M}_E$ and $m_{E/v}(A):=m(A\cup\{v\})$ for all
$A\subseteq E\setminus\{v\}$. It is easy to see that they are in fact
two arithmetic matroids.

The \textit{arithmetic Tutte polynomial}
$M_E(x,y)=M(\mathfrak{A}_E;x,y)$ of an arithmetic matroid
$\mathfrak{A}_E=(\mathfrak{M}_E,m)$ is defined as
$$M_E(x,y):=\sum_{A\subseteq
E}m(A)(x-1)^{\rk(E)-\rk(A)}(y-1)^{|A|-\rk(A)}.$$ M. D'Adderio and L.
Moci provided a combinatorial interpretation of the arithmetic Tutte
polynomial of any arithmetic matroid, showing in particular the
positivity of its coefficients~\cite[Theorem 5.1]{Moci2011}. Similar
to the Tutte polynomials, the arithmetic Tutte polynomials satisfy
the following deletion-contraction recurrence relations
\begin{equation}
M_E(x,y)=
\begin{cases}
M_{E\setminus v}(x,y)+M_{E/v}(x,y), & \text{if $v$ is proper,} \\
(x-1)M_{E\setminus v}(x,y)+M_{E/v}(x,y), & \text{if $v$ is a coloop,} \\
M_{E\setminus v}(x,y)+(y-1)M_{E/v}(x,y), & \text{if $v$ is a loop,}
\end{cases}
\end{equation}
where $v$ is an element of $E$, $M_{E\setminus v}(x,y)$ is the
arithmetic Tutte polynomial of $\mathfrak{A}_{E\setminus v}$, and
$M_{E/v}(x,y)$ is the arithmetic Tutte polynomial of
$\mathfrak{A}_{E/v}$.

The \textit{characteristic polynomial} $\chi(\mathfrak{A}_E;t)$ of
an arithmetic matroid $\mathfrak{A}_E=(\mathfrak{M}_E,m)$ is defined
as
$$\chi(\mathfrak{A}_E;t):=\sum_{A\subseteq
E}(-1)^{|A|}m(A)t^{\rk(E)-\rk(A)}.$$ It is easy to see that
$\chi(\mathfrak{A}_E;t)=(-1)^{\rk(E)}M_E(1-t,0)$ and the
characteristic polynomial satisfies the following
deletion-contraction recurrence relations
\begin{equation}
\chi(\mathfrak{A}_E;t)=
\begin{cases}
\chi(\mathfrak{A}_{E\setminus v};t)-\chi(\mathfrak{A}_{E/v};t), & \text{if $v$ is proper or a loop,} \\
t\cdot\chi(\mathfrak{A}_{E\setminus
v};t)-\chi(\mathfrak{A}_{E/v};t), & \text{if $v$ is a coloop.}
\end{cases}
\end{equation}

\begin{theorem}
Let $\mathfrak{A}_E=(\mathfrak{M}_E,m)$ be an arithmetic matroid of
rank $r=\rk(E)$ and its characteristic polynomial
$$\chi(\mathfrak{A}_E;t)=a_0t^r+a_1t^{r-1}+a_2t^{r-2}+\cdots+a_{r-1}t+a_r.$$
Then $(-1)^{k}a_k\geq0$ and $(-1)^{k}\sum_{i=0}^{k}a_i\geq0$, for
$0\leq k\leq r$.
\end{theorem}

\begin{proof}
We shall argue by induction on $|E|$. If $|E|=0$, then the result is
immediate. Suppose the result holds for $|E|\leq m$ and let
$|E|=m+1$. If all elements of $E$ are loops, then
$$\chi(\mathfrak{A}_E;t)=\sum_{A\subseteq E}(-1)^{|A|}m(A)=a_0.$$
Since $\rk(\emptyset)=\rk(E)=0$, so $a_0=\rho(\emptyset,E)\geq0$.
Otherwise, there exists an element $v\in E$, which is not a loop.\\
Case I. If $v$ is proper, then $\rk_{E\setminus
v}(E\setminus\{v\})=\rk(E)$ and $\rk_{E/v}(E\setminus\{v\})=\rk(E)-1$.
We can write $$\chi(\mathfrak{A}_{E\setminus
v};t)=b_0t^r+b_1t^{r-1}+b_2t^{r-2}+\cdots+b_{r-1}t+b_r$$ and
$$\chi(\mathfrak{A}_{E/v};t)=c_0t^{r-1}+c_1t^{r-2}+c_2t^{r-3}+\cdots+c_{r-2}t+c_{r-1}.$$
Using the deletion-contraction recurrence
$\chi(\mathfrak{A}_E;t)=\chi(\mathfrak{A}_{E\setminus
v};t)-\chi(\mathfrak{A}_{E/v};t)$, we have $a_0=b_0$ and
$a_k=b_k-c_{k-1}$, for $1\leq k\leq r$. Then by the induction
assumption, $a_0=b_0\geq0$,
$(-1)^ka_k=(-1)^kb_k+(-1)^{k-1}c_{k-1}\geq0$ and
\begin{align*}
(-1)^k\sum_{i=0}^ka_i & = (-1)^k\Big(b_0+\sum_{i=0}^k(b_i-c_{i-1})\Big)\\
                      & =
                      (-1)^k\sum_{i=0}^kb_i+(-1)^{k-1}\sum_{i=0}^{k-1}c_i\geq0,
\end{align*}
for $1\leq k\leq r$.\\
Case II. If $v$ is a coloop, then $\rk_{E\setminus
v}(E\setminus\{v\})=\rk(E)-1$ and
$\rk_{E/v}(E\setminus\{v\})=\rk(E)-1$. We can write
$$\chi(\mathfrak{A}_{E\setminus
v};t)=b_0t^{r-1}+b_1t^{r-2}+b_2t^{r-3}+\cdots+b_{r-2}t+b_{r-1}$$ and
$$\chi(\mathfrak{A}_{E/v};t)=c_0t^{r-1}+c_1t^{r-2}+c_2t^{r-3}+\cdots+c_{r-2}t+c_{r-1}.$$
Using the deletion-contraction recurrence
$\chi(\mathfrak{A}_E;t)=t\cdot\chi(\mathfrak{A}_{E\setminus
v};t)-\chi(\mathfrak{A}_{E/v};t)$, we have $a_0=b_0$, $a_r=-c_{r-1}$
and $a_k=b_k-c_{k-1}$, for $1\leq k\leq r-1$. Then by the induction
assumption, $a_0=b_0\geq0$, $(-1)^ra_r=(-1)^{r-1}c_{r-1}\geq0$,
$(-1)^ka_k=(-1)^kb_k+(-1)^{k-1}c_{k-1}\geq0$ and
\begin{align*}
(-1)^k\sum_{i=0}^ka_i & = (-1)^k\Big(b_0+\sum_{i=0}^k(b_i-c_{i-1})\Big)\\
                      & =
                      (-1)^k\sum_{i=0}^kb_i+(-1)^{k-1}\sum_{i=0}^{k-1}c_i\geq0,
\end{align*}
for $1\leq k\leq r-1$. Together with
$(-1)^r\sum_{i=0}^ra_i=(-1)^r\chi(\mathfrak{A}_E;1)=M_E(0,0)\geq0$,
the result holds for the coefficients of the characteristic
polynomial of $\mathfrak{A}_E$.\\
Therefore, by induction, $(-1)^{k}a_k\geq0$ and
$(-1)^{k}\sum_{i=0}^{k}a_i\geq0$, for $0\leq k\leq r$.
\end{proof}

Let $E$ be a finite list of elements of a finitely generated abelian
group $G$. Given a sublist $A\subseteq E$, let $\langle A\rangle$ be
the subgroup of $G$ generated by the elements of $A$. The rank
$\rk(A)$ is defined as the rank of the finitely generated abelian
group $\langle A\rangle$ and the multiplicity $m(A)$ is defined as
the cardinality of the torsion subgroup of $G/\langle A\rangle$.
Then $\mathfrak{A}_E=(\mathfrak{M}_{E}, m)=(E,\rk,m)$ is an
arithmetic matroid. Given an arithmetic matroid
$\mathfrak{A}_E=(\mathfrak{M}_{E}, m)$, if it is
\textit{representable}, i.e., $E$ can be regarded as a list of
elements in a finitely generated abelian group $G$ and we can always
assume that $\rk(E)$ is equal to the rank of $G$, then we have the
following \textit{generalized toric arrangement}
$\mathcal{T}_E:=\{H_{\lambda}:\lambda\in E\}$ defined by $E$ on
$T:=\text{Hom}(G,\mathbb{C}^*)$, where $H_{\lambda}:=\{t\in
T:\lambda(t)=1\}$ (see \cite{Moci2011}). The \textit{characteristic
polynomial} of the toric arrangement $\mathcal{T}_E$ is defined as
$$\chi(\mathcal{T}_E;t):=\sum_{C\in
\mathcal{C}(\mathcal{T}_E)}\mu(T_C,C)t^{\text{dim}\ C},$$ where
$\mathcal{C}(\mathcal{T}_E)$ is the set of all the connected
components of all possible nonempty intersections of the
subvarieties $H_{\lambda}$, ordered by reverse inclusion with the
M\"obius function $\mu$ and $T_C$ is the connected component of $T$
that contains $C$. L. Moci~\cite{Moci2012} proved that
$$\chi(\mathcal{T}_E;t)=(-1)^{\rk(E)}M_E(1-t,0)=\chi(\mathfrak{A}_E;t).$$ Hence, we have the following Corollary.

\begin{corollary}
Let $\mathfrak{A}_E=(\mathfrak{M}_E,m)$ be a representable
arithmetic matroid of rank $r=\rk(E)$ and the characteristic
polynomial of the associated toric arrangement $\mathcal{T}_E$
$$\chi(\mathcal{T}_E;t)=a_0t^r+a_1t^{r-1}+a_2t^{r-2}+\cdots+a_{r-1}t+a_r.$$
Then $(-1)^{k}a_k\geq0$ and $(-1)^{k}\sum_{i=0}^{k}a_i\geq0$, for
$0\leq k\leq r$.
\end{corollary}

\section{Discussions and Problems}
Let $\mathcal{A}$ be an $n$-dimensional arrangement of $m$
hyperplanes and rank $r$. If all hyperplanes of $\mathcal{A}$ are
restricted to the $r$-dimensional subspace spanned by their normal
vectors, we will obtain an $r$-dimensional arrangement of $m$
hyperplanes and rank $r$, whose characteristic polynomial has the
same coefficient sequence as $\chi(\mathcal{A};t)$. Under this
restriction, the characteristic polynomial becomes $t$-free, i.e,
containing no $t$ as a factor. So we can assume, throughout this
section, the hyperplane arrangement $\mathcal{A}$ is essential,
i.e., the characteristic polynomial of $\mathcal{A}$ is
\begin{equation}\label{formula-chi}
\chi(\mathcal{A};t)=a_0t^r-a_1t^{r-1}+\cdots+(-1)^ra_r.
\end{equation}
Recall the no broken circuit Theorem \ref{theorem-BC} that $a_k$
counts the number of $\chi$-independent $k$-subsets of
$\mathcal{A}$. It is then obvious that $a_k\le {m\choose k}$. On the
other hand, $a_r\ne 0$ implies that there exists at least one
$\chi$-independent $r$-subset $\mathcal{B}$ of $\mathcal{A}$. Note
the fact from the definition that any subset of a $\chi$-independent
set is still $\chi$-independent. Then all subsets of $\mathcal{B}$
are $\chi$-independent, which implies $a_k\ge {r\choose k}$ by the
no broken circuit Theorem \ref{theorem-BC}. In this sense, the
inequality ${r\choose k}\le a_k\le {m\choose k}$ of Corollary
\ref{corollary-3} can be easily obtained from the no broken circuit
theorem. Note that from the form of formula (\ref{formula-main}),
Theorem \ref{theorem-main} can be viewed as a generalization of
Corollary \ref{corollary-3}. It leads us to the following problem,
whether the formula (\ref{formula-main}) can be directly obtained
from the no broken circuit theorem or not. To this purpose, we
introduce an operation $\D$ of the characteristic polynomial
$\chi(\mathcal{A};t)$,
\[\D \chi(\mathcal{A};t)=\frac{\chi(\mathcal{A};t)-\chi(\mathcal{A};1)}{t-1}.\]
Denote $\D^0$ the identity operation and $\D^i
\chi(\mathcal{A};t)=\underbrace{\D\circ\cdots\circ\D}_i
\chi(\mathcal{A};t)$ for $i\in \Bbb{N}$. Note that
\begin{eqnarray*}
\chi(\mathcal{A};t)-\chi(\mathcal{A};1)&=&a_0(t^r-1)-a_1(t^{r-1}-1)+\cdots+(-1)^{r-1}a_{r-1}(t-1)\\
&=&(t-1)\left(\sum_{k=0}^{r-1}t^{r-1-k}\sum_{i=0}^{k}(-1)^ia_i\right).
\end{eqnarray*}
Then we have
\begin{eqnarray*}
\D
\chi(\mathcal{A};t)=\sum_{k=0}^{r-1}\left(\sum_{i=0}^{k}(-1)^ia_i\right)t^{r-1-k}.
\end{eqnarray*}
In general, for any non-positive integer $q$, we have
\begin{eqnarray*}
\D^{-q}
\chi(\mathcal{A};t)&=&\sum_{k=0}^{r+q}(-1)^k\left(\sum_{i=0}^{k}{q\choose
k-i}a_{i}\right)t^{r+q-k}.
\end{eqnarray*}
If $\D^{-q} \chi(\mathcal{A};t)$ can be geometrically realized as
the characteristic polynomial of an arrangement of $m+q$ hyperplanes
and rank $r+q$, from discussions at the beginning of this section,
the inequality (\ref{formula-main}) can be easily obtained from the
no broken circuit theorem, which answers the previous question. So
our question is reduced to finding a geometric realization of
$\D^{-q} \chi(\mathcal{A};t)$ as the characteristic polynomial of
some arrangement of hyperplanes for any $q\in \Bbb{Z}_{\le 0}$.

Furthermore, it can be further reduced to finding a geometric
realization of $\D \chi(\mathcal{A};t)$ for any $\mathcal{A}$, i.e.,
$\D\chi(\mathcal{A};t)=\chi(\mathcal{A}_1;t)$ for some hyperplane
arrangement $\mathcal{A}_1$. Similarly, we shall have a hyperplane
arrangement $\mathcal{A}_2$ such that
$\D\chi(\mathcal{A}_1;t)=\chi(\mathcal{A}_2;t)$. Continuing this
process, the geometric realization of $\D^{-q}\chi(\mathcal{A};t)$
can be obtained reductively for all $q\in \Bbb{Z}_{\le 0}$.

When all hyperplanes in $\mathcal{A}$ pass through the origin
(called a linear hyperplane arrangement), we are able to construct
an affine hyperplane arrangement $\d\mathcal{A}$ such that
$\chi(\d\mathcal{A};t)=\D\chi(\mathcal{A};t)$. Suppose $\mathcal{A}$
is a linear arrangement of $m+1$ hyperplanes in $\Bbb{R}^n$. Given
$K_0\in \mathcal{A}$ with the defining equation $K_0:
\sum_{i=1}^n\alpha_ix_i=0$, the \emph{deconing} $\d\mathcal{A}$ of
$\mathcal{A}$ is an arrangement of $m$ hyperplanes in the affine
space $K_1:\sum_{i=1}^n\alpha_ix_i=1$, which is defined by
\[\d\mathcal{A}=\{H\cap K_1\mid H\in \mathcal{A}, H\neq K_0\}.\]
Since $\chi(\mathcal{A},1)=0$ when $\mathcal{A}$ is linear, we have
\[\chi(\d\mathcal{A};t)=\D\chi(\mathcal{A};t).\]
Namely, for the linear hyperplane arrangement $\mathcal{A}$, the
deconing construction can geometrically realize
$\D\chi(\mathcal{A};t)$ as the characteristic polynomial of the
hyperplane arrangement $\d\mathcal{A}$. However, it is not easy to
find a generalization of this construction directly for affine
hyperplane arrangements.

To simplify our problem, we introduce another construction, called
{\emph coning}, which is the inverse operation of deconing. If a
hyperplane $H$ in $\Bbb{R}^n$ is defined by $H:a_1x_1+\cdots
+a_nx_n=b$, let $\c H$ be a hyperplane in $\Bbb{R}^{n+1}$ defined to
be $\c H:a_1x_1+\cdots +a_nx_n=bx_{n+1}$. The coning $\c\mathcal{A}$
of $\mathcal{A}$ is a linear hyperplane arrangement in
$\Bbb{R}^{n+1}$ consisting of all $\c H$ with $H\in \mathcal{A}$ and
an extra hyperplane $K:x_{n+1}=0$. For the coning construction, we
have
\[
\chi(\c \mathcal{A}; t)=(t-1)\chi(\mathcal{A}; t).
\]

For any essential hyperplane arrangement $\mathcal{A}$ with the
characteristic polynomial
$\chi(\mathcal{A};t)=a_0t^r-a_1t^{r-1}+\cdots+(-1)^ra_r$, if there
exists a hyperplane $\mathcal{B}$ such that
$\chi(\mathcal{B};t)=\D\chi(\mathcal{A};t)$, then except for the
constant term, the coefficients of the characteristic polynomial
$\chi(\c\mathcal{B};t)$ are the same as $\chi(\mathcal{A};t)$, i.e.,
\[
\chi(\c\mathcal{B};t)=a_0t^r-a_1t^{r-1}+\cdots+(-1)^{r-1}a_{r-1}t-
[a_0-a_1+\cdots+(-1)^{r-1}a_{r-1}],
\]
Conversely, suppose there exists a linear hyperplane arrangement
$\mathcal{C}$ such that except for the constant term, all
coefficients of its characteristic polynomial are the same as
$\mathcal{A}$. Then we have
\[
\chi(\d\mathcal{C};t)=\D\chi(\mathcal{C};t)=\D\chi(\mathcal{A};t).
\]
Hence, the above question concerning the interpretation of the formula (\ref{formula-main}) is finally reduced to the following problem.\\
\vspace{-.3cm}\\
{\bf Problem.} Given an essential hyperplane arrangement
$\mathcal{A}$, is there a linear hyperplane arrangement
$\mathcal{B}$ whose characteristic polynomial is
$\chi(\mathcal{B};t)=\chi(\mathcal{A}; t)-\chi(\mathcal{A};1)$?\\

When we consider the characteristic polynomial of matroids, the problem can be stated as follows. Let $\mathfrak{M}_E=(E, \rk)$ be a matroid and $r=\rk(E)$, whose characteristic polynomial is written (abuse of notation) to be
\[
\chi(\mathfrak{M}_E;t)=a_0t^r-a_1t^{r-1}+\cdots+(-1)^ra_r..
\]
Note that $\chi(\mathfrak{M}_{E};1)=0$ which is easily seen from the deletion-contraction recurrence. Similarly, we can define
\[
\D\chi(\mathfrak{M}_{E};t)=\frac{\chi(\mathfrak{M}_E;t)}{t-1}.
\]
Our problem is whether there is an matroid $\mathfrak{M}_{E^\prime}=(E^\prime, \rk^\prime)$ such that
\[
\chi(\mathfrak{M}_{E^\prime};t)=\D\chi(\mathfrak{M}_E;t)-\D\chi(\mathfrak{M}_E;1).
\]
In the following, we give a sufficient condition for the matroid $\mathfrak{M}_E$ such that the above problem holds.
\begin{proposition}
Let $\mathfrak{M}_E=(E, \rk)$ be a matroid with $\rk(E)=r$. Let $\mathcal{C}$ be its circuit structure, i.e, the collection of all circuits of $\mathfrak{M}_E$. If there is an element $v\in E$ such that every circuit $C\in \mathcal{C}$ containing $v$ is of size $r+1$, and the contraction $\mathfrak{M}_{E/v}=(E\setminus\{v\}, \rk_{E/v})$ with respect to $v$ has the circuit structure
\[
\mathcal{C}^\prime=\{C\setminus \{v\}\mid C\in \mathcal{C}\},
\]
then we have
\[
\chi(\mathfrak{M}_{E/v};t)=\D\chi(\mathfrak{M}_E;t)-
\D\chi(\mathfrak{M}_E;1).
\]
\end{proposition}
\begin{proof}Given a linear order  on $E$ with the maximal element $v$.
Write
\[
\chi(\mathfrak{M}_{E/v};t)=b_0t^{r-1}-b_1t^{r-2}+\cdots+(-1)^{r-1}b_{r-1}.
\]
Note that $\chi(\mathfrak{M}_{E/v};1)=0$ and
\begin{eqnarray*}
\D\chi(\mathfrak{M}_E;t)&=&\frac{\chi(\mathfrak{M}_E;t)}{t-1}\\
&=&a_0t^{r-1}-(a_1-a_0)t^{r-2}+(a_2-a_1+a_0)t^{r-3}+\cdots
\end{eqnarray*}
It suffices to show that $a_i=b_i+b_{i-1}$ for $0\le i\le r-2$ (assuming $b_{-1}=0$).
Recall that a broken circuit is a subset obtained by removing the maximal element of a circuit under the given order of $E$. The classical no broken circuit theorem on matroids tells $a_i=|\Omega_i|$, where
\[
\Omega_i=\{S\subseteq E\mid |S|=i, S\text{ contains no broken circuit of } \mathfrak{M}\}.
\]
Let $\Omega_i=\Omega_{i,0}\sqcup\Omega_{i,v}$, where
\[
\Omega_{i,0}=\{S\in \Omega_i\mid v\notin S\}\quad\And\quad\Omega_{i,v}=\{S\in \Omega_i\mid v\in S\}.
\]
Similarly, we have $b_i=|\Omega_i^\prime|$ and $b_{i-1}=|\Omega_{i-1}^\prime|$, where
\[
\Omega_i^\prime=\{T\subseteq E\setminus \{v\}\mid |T|=i, T\text{ contains no broken circuit of } \mathfrak{M}_{E/v}\}.
\]
It remains to show that $\Omega_{i,0}=\Omega_i^\prime$ for $i\le r-2$, and $\Omega_{i-1}^\prime=\{S\setminus \{v\}\mid S\in \Omega_{i,v}\}$ for $i\le r-1$.

First we prove $\Omega_{i,0}= \Omega_i^\prime$ for $i\le r-2$. For any $S\in \Omega_{i,0}$, suppose $S\notin \Omega_i^\prime$, i.e., there is an element $e\in E\setminus \{v\}$ such that $S\sqcup\{e\}$ contains a circuit $C$ of $\mathfrak{M}_{E/v}$. From the assumption on the circuit structure of $\mathfrak{M}_{E/v}$, either $C$ or $C\sqcup\{v\}$ is a circuit of $\mathfrak{M}_E$. Since $S$ contains no broken circuits, it forces $C\sqcup\{v\}$ to be a circuit of $\mathfrak{M}_E$ and then $|C|=r$, a contradiction to $|S|=i\le r-2$. Conversely, for any $T\in\Omega_i^\prime$, suppose there is an element $e\in E$ such that $T\sqcup \{e\}$ contains a circuit of $\mathfrak{M}$. We have $v\neq e$. Otherwise, $T$ contains a circuit of $\mathfrak{M}_{E/v}$ from the assumption. Hence $T\sqcup \{e\}\subseteq E\setminus\{v\}$ contains a circuit of $\mathfrak{M}_{E/v}$ which is a contradiction to $T\in\Omega_i^\prime$.

Next we prove $\Omega_{i-1}^\prime=\{S\setminus \{v\}\mid S\in \Omega_{i,v}\}$ for $i\le r-1$. For any $T\in\Omega_{i-1}^\prime$, suppose $T\sqcup\{v\}\notin \Omega_{i,v}$, i.e., there is an element $e\in E$ such $T\sqcup\{e,v\}$ contains a circuit $C$ of $\mathfrak{M}_E$. We have $v\in \mathcal{C}$. Otherwise, $T\sqcup\{e\}$ contains the circuit $C$ of $\mathfrak{M}_{E/v}$, a contradiction to $T\in\Omega_{i-1}^\prime$.  Thus the circuit $C$ must have size $r+1$ by the assumption of the circuit structure on $\mathfrak{M}_E$. However, $|T|=i-1\le r-2$ for $i\le r-1$, a contradiction. Hence we obtain $T\sqcup\{v\}\in \Omega_{i,v}$. Conversely, if $S\in \Omega_{i,v}$ for $i\le r-1$, suppose $S\setminus \{v\}\notin \Omega_{i-1}^\prime$, i.e, there is an element $e$ of $E\setminus \{v\}$ such that $S\sqcup \{e\}\setminus \{v\}$ contains a circuit $C$ of $\mathfrak{M}_{E/v}$. From the assumption on the circuit structure of $\mathfrak{M}_{E/v}$, either $C$ or $C\sqcup\{v\}$ is a circuit of $\mathfrak{M}_E$. If $C$ is a circuit of $\mathfrak{M}_E$, then $S$ contains a broken circuit of $\mathfrak{M}_E$, a contradiction to $S\in \Omega_{i,v}$. If $C\sqcup\{v\}$ is a circuit of $\mathfrak{M}_E$, from the assumption on the circuit structure of $\mathfrak{M}_E$, we have $|C|=r$. It implies $|S|\ge r$, a contradiction to $S\in \Omega_{i,v}$ for $i\le r-1$.

\end{proof}

\end{document}